\documentclass{article}
\usepackage{amssymb}
\usepackage{amsfonts}
\usepackage{amsmath}

\newtheorem{theorem}{Theorem}

\newtheorem{corollary}[theorem]{Corollary}

\newtheorem{definition}[theorem]{Definition}
\newtheorem{example}[theorem]{Example}

\newtheorem{lemma}[theorem]{Lemma}

\newtheorem{proposition}[theorem]{Proposition}

\newenvironment{proof}[1][Proof]{\noindent\textbf{#1.} }{\ \rule{0.5em}{0.5em}}
\input{tcilatex}
\begin{document}

\title{Strictly periodic points and periodic factors of cellular automata }
\author{Nacira Allaoua and Rezki Chemlal. \and Laboratoire de Math\'{e}%
matiques Appliqu\'{e}es, Facult\'{e} des sciences exactes. \and Universit%
\'{e} Abderahmane Mira Bejaia.06000 Bejaia Algeria.}
\maketitle

\begin{abstract}
We show that the set of strictly temporally periodic points of cellular
automata with almost equicontinuous points is dense in the topological
support of the measure. This extends a result of Lena, Margara and Dennunzio
about the density of the set of strictly temporally periodic of cellular
automata with equicontinuous points.


\end{abstract}

\section{Introduction}

Cellular automata are computational objects used and studied as models in
many applied domains.

They are used to model complex systems, for example in statistical physics
as microscale models of hydrodynamics \cite{Chop03}, to model social choices 
\cite{Nowak96} or tumor growth \cite{ABM03}.

This interest is due to the enormous potential they hold in modeling complex
systems.

A cellular automaton is made of an infinite lattice of finite identical
automata. The lattice is usually $\mathbb{Z}^{n}$ with $n$ called the
dimension of the cellular automaton. Each automaton updates its state
synchronously according to a local rule on the basis of its actual state and
of the one of a fixed finite set of neighboring automata. The set of
possible states of an automaton is called the alphabet and each element of
the alphabet is referred to as a letter. A configuration is a snapshot of
the state of all automata in the lattice.

The study of CA as symbolic dynamical systems began with Hedlund \cite{Hed69}%
. Dynamical behavior of cellular automata is studied mainly in the context
of discrete dynamical systems by equipping the space of configurations with
the product topology which makes it homeomorphic to the Cantor space.

\section{Definitions and background.}

\subsection{Topological dynamics.}

A topological dynamical system $\left( X,T\right) $ consists of a compact
metric space $X$ and a continuous self--map $T$ .

A point $x$ is said periodic if there exists $p>0$ with $T^{p}\left(
x\right) =x.$ The least $p$ with this property is called the period of $x.$
A\ point $x$ is eventually or ultimately periodic if $T^{m}\left( x\right) $
is periodic for some $m\geq 0.$

In the same way a dynamical system is said periodic if there exists $p>0$
with $T^{p}\left( x\right) =x$ for every $x\in X$ and eventually periodic if 
$T^{m}$ is periodic for some $m\geq 0.$

A point $x\in X$ is said to be an equicontinuity point, or to be Lyapunov
stable, if for any $\epsilon >0$, there exists $\delta >0$ such that if $%
d\left( x,y\right) <\delta $ one has $d\left( T^{n}\left( y\right)
,T^{n}\left( x\right) \right) <\epsilon $ for any integer $n\geq 0$.

When all points of $\left( X,T\right) $ are equicontinuity points $\left(
X,T\right) $ is said to be equicontinuous: since $X$ is compact an
equicontinuous system is uniformly equicontinuous. We say that $\left(
X,T\right) $ is almost equicontinuous if the set of equicontinuous points is
residual.

We say that $\left( X,T\right) $ is sensitive if for any $x\in X$ we have :%
\begin{equation*}
\exists \epsilon >0,\forall \delta >0,\exists y\in B_{\delta }\left(
x\right) ,\exists n\geq 0\text{ }\mathrm{such}\text{ \textrm{that }}d\left(
T^{n}\left( y\right) ,T^{n}\left( x\right) \right) \geq \epsilon .
\end{equation*}

We say that $\left( X,T\right) $ is expansive if we have :%
\begin{equation*}
\exists \epsilon >0,\forall x\neq y\in X,\exists n\geq 0,d\left( T^{n}\left(
x\right) ,T^{n}\left( y\right) \right) \geq \epsilon .
\end{equation*}

A dynamical system $\left( X,T\right) $\emph{\ }is transitive if for any
nonempty open sets $U,V\subset A^{\mathbb{Z}}$ there exists $n>0$ with $%
U\cap F^{-n}\left( V\right) \neq \emptyset .$ It is said topologically
mixing if for any nonempty open sets $U,V\subset A^{\mathbb{Z}},U\cap
F^{-n}\left( V\right) \neq \emptyset $ for all sufficiently large $n.$

\subsection{Factors}

Let $(X,T)$ and $(Y,U)$ two dynamical systems and $\pi $ a continuous
surjective map $\pi :X\rightarrow Y$ such that $\pi \circ T=U\circ \pi .$We
say that $\pi $ is a factor map and $(Y,U)$ is a factor of $(X,T).$If $\pi $
is bijective, we say that $\pi $ is a conjugacy and that $(X,T)$ and $(Y,U)$
are conjugate. The conjugacy is an equivalence relation on dynamical systems.%
\begin{equation*}
\begin{tabular}{lll}
$X$ & $\overset{T}{\longrightarrow }$ & $X$ \\ 
$\pi \downarrow $ &  & $\downarrow \pi $ \\ 
$Y$ & $\overset{U}{\longrightarrow }$ & $Y$%
\end{tabular}%
\end{equation*}

The factor relation is transitive in the sense that if $\left( X,T\right)
,\left( Y,U\right) ,\left( Z,V\right) $ are three dynamical systems such
that $\left( Y,U\right) $ is a factor of $\left( X,T\right) $ and $\left(
Z,V\right) $ is a factor of $\left( Y,U\right) $\ then $\left( Z,V\right) $
is a factor of $\left( X,T\right) .$

From the definition the factor of a surjective dynamical system is still
surjective. If $\left( X,T\right) $ is periodic of period $p$ and $\left(
Y,U\right) $ is a factor of $\left( X,T\right) $ are two dynamical systems
such that $\left( Y,U\right) $ is a factor of $\left( X,T\right) $. If $%
\left( X,T\right) $ and $\left( Y,U\right) $ are periodic of respectively
periods $p$ and $q$ then $q$ is a divisor of $p.$

\begin{proposition}
Any dynamical system $\left( X,T\right) $ possesses a maximal equicontinuous
factor $\pi :\left( X,F\right) \rightarrow \left( Y,U\right) ,$ such that
for any equicontinuous factor $\varphi :\left( X,T\right) \rightarrow \left(
Z,V\right) $ there is a unique factor map $\psi :\left( Y,U\right)
\rightarrow \left( Z,V\right) $ with $\psi \circ \pi =\varphi .$ The system $%
\left( Y,U\right) $ is unique up to conjugacy.
\end{proposition}

\subsection{Symbolic dynamics}

Let $A$ be a finite set called the \emph{alphabet}. A word is a any finite
sequence of elements of $A$.Denote by $A^{\ast }=\cup _{n\in \mathbb{N}%
^{\ast }}A^{n}$ the set of all finite words $u=u_{0}...u_{n-1};$the length
of a word $u\in A^{n}$ is $\left\vert u\right\vert =n.$

Let $A^{\mathbb{Z}}$ denote the set of all functions $x:\mathbb{Z\rightarrow 
}A,$ which we regard as $\mathbb{Z}$-indexed \emph{configurations} of
elements in $A$.

We write such a configuration as $x=\left( x_{n}\right) _{n\in \mathbb{Z}},$
where $x_{n}\in A$ for all $n\in \mathbb{Z}$, and refer to $A^{\mathbb{Z}}$
as \emph{configuration space}.

Treat $A$ as a discrete topological space; then A is compact, so $A^{\mathbb{%
Z}}$ is compact in the Tychonoff product topology. In fact, $A^{\mathbb{Z}}$
is a Cantor space: it is compact, perfect,totally disconnected, and
metrizable.

The standard metric on $A^{\mathbb{Z}}$ is defined by

\begin{equation*}
d\left( x,y\right) =2^{-n}with\text{ }n=mini\geq 0:x_{i}\neq y_{i}\,\mathrm{%
or}\,x_{-i}\neq y_{-i}
\end{equation*}

Let $x$ a configuration of $A^{\mathbb{Z}},$ for two integers $i,j$ with $%
i<j $ we denote by $x\left( i,j\right) \in A^{j-i+1}$ the word $%
x_{i}...x_{j}.$

For any word $u$ we define the cylinder $\left[ u\right] _{l}=\left\{ x\in
A^{\mathbb{Z}}:x\left( l,l+\left\vert u\right\vert -1\right) =u\right\} $
where the word $u$ is at the position $l.$ The cylinder $\left[ u\right]
_{0} $ is simply noted $\left[ u\right] $. The cylinders are clopen (closed
open) sets.

The metric $d$ is non archimedian; that is $d\left( x,y\right) \leq \max
\left\{ d\left( x,z\right) ,d\left( z,y\right) \right\} $. Consequently any
two cylinders either one contain the other or they intersect trivially.

The shift map $\sigma :$ $A^{\mathbb{Z}}\rightarrow $ $A^{\mathbb{Z}}$ is
defined as $\sigma \left( x\right) _{i}=x_{i+1},$ for any $x\in A^{\mathbb{Z}%
}$ and $i\in \mathbb{Z}$. The shift map is a continuous and bijective
function on $A^{\mathbb{Z}}.$ The dynamical system $\left( A^{\mathbb{Z}%
},\sigma \right) $ is commonly called \emph{full shift.}

The configuration $^{\infty }u^{\infty }$ is defined by $\left( ^{\infty
}u^{\infty }\right) _{k\left\vert u\right\vert +i}=u_{i}$ for $k\in \mathbb{Z%
}$ for $k\in \mathbb{Z}$, $0\leq i<\left\vert u\right\vert $ and $u\in
A^{\ast }.$ The configuration $^{\infty }u^{\infty }$ is shift periodic and
is called \emph{spatially periodic} configuration.

A cellular automaton is a continuous map $F:A^{\mathbb{Z}}\rightarrow A^{%
\mathbb{Z}}$ commuting with the shift. By the Curtis-Hedlund Lyndon theorem 
\cite{Hed69} for every cellular automaton $F$ there exist an integer $r$ and
a block map $f$ from $A^{2r+1}$ to $A$ such that $F\left( x\right)
_{i}=f\left( x_{i-r},...,x_{i},...x_{i+r}\right) .$ The integer $r$ is
called the radius of the cellular automaton.

Endowed with the sigma-algebra on $A^{\mathbb{Z}}$ generated by all cylinder
sets and $\nu $ the uniform measure which gives the same probability to
every letter of the alphabet. The uniform measure is invariant if and only
if the cellular automaton is surjective \cite{Hed69}.

\subsection{Equicontinuous and almost equicontinuous points of cellular
automata}

\subsubsection{K\r{u}rka's classification}

K\r{u}rka \cite{Kur03} introduced a topological classification based on the
equicontinuity, sensitiveness and expansiveness properties. The existence of
an equicontinuous point is equivalent to the existence of a blocking word
i.e. a configuration that stop the propagation of the perturbations on the
one dimensional lattice.

\begin{definition}
Let $F$ be a cellular automaton.\newline
1. A point $x$ is an equicontinuous point if : 
\begin{equation*}
\forall \epsilon >0,\exists \delta >0,\forall y:d\left( x,y\right) <\delta
,\forall n\geq 0,d\left( F^{n}\left( y\right) ,F^{n}\left( x\right) \right)
<\epsilon .
\end{equation*}%
\newline
2. We say that $F$ is equicontinuous if every point $x\in A^{\mathbb{Z}}$ is
an equicontinuous point.\newline
3.We say that $F$ is sensitive if for all $x\in A^{\mathbb{Z}}$ we have : 
\begin{equation*}
\exists \epsilon >0,\forall \delta >0,\exists y:d\left( x,y\right) <\delta
,\exists n\geq 0\text{,}d\left( F^{n}\left( y\right) ,F^{n}\left( x\right)
\right) \geq \epsilon .
\end{equation*}
\end{definition}

\begin{definition}
Let $F$ be a cellular automaton. A word $w$ with $\left\vert w\right\vert
\geq s$ is an $s$-blocking word for $F$ if there exists $p\in \left[
0,\left\vert w\right\vert -s\right] $ such that for any $x,y\in \left[ w%
\right] $ we have $F^{n}\left( x\right) \left( p,p+s\right) =F^{n}\left(
y\right) \left( p,p+s\right) $ for all $n\geq 0.$
\end{definition}

\begin{proposition}
Let $F$ be a cellular automaton\ with radius $r>0.$ The following conditions
are equivalent.\newline
1. $F$ is not sensitive. \newline
2. $F$ has an $r-$blocking word.\newline
3. $F$ has some equicontinuous point.
\end{proposition}

\subsubsection{Gilman's classification}

Based on the Wolfram's work \cite{Wol84}, Gilman \cite{Gil87} introduced a
classification using Bernoulli measures which are not necessarily invariant.
He also introduced the concepts of measurable equicontinuous point and of
measurable expansivity.

Cellular automata can then be divided into three classes : cellular automata
with equicontinuous points, with almost equicontinuous points but without
equicontinuous points and almost expansive cellular automata.

In \cite{Tis08} Tisseur extends the Gilman's classification to any shift
ergodic measure and gives an example of a cellular automaton with an
invariant measure which have almost equicontinuous points but without
equicontinuous points.

\begin{definition}
Let $F$ be a cellular automaton and $\left[ i_{1},i_{2}\right] $ a finite
interval of $\mathbb{Z}$. For $x\in A^{\mathbb{Z}}$. We define $B_{\left[
i_{1},i_{2}\right] }\left( x\right) $ by : 
\begin{equation*}
B_{\left[ i_{1},i_{2}\right] }\left( x\right) =\left\{ y\in A^{\mathbb{Z}%
},\forall j:F^{j}\left( x\right) \left( i_{1},i_{2}\right) =F^{j}\left(
y\right) \left( i_{1},i_{2}\right) \right\} .
\end{equation*}
\end{definition}

For any interval $\left[ i_{1},i_{2}\right] $ the relation $\mathfrak{R}$
defined by $x\mathfrak{R}y$ if and only if $\forall j:F^{j}\left( x\right)
\left( i_{1},i_{2}\right) =F^{j}\left( y\right) \left( i_{1},i_{2}\right) $
is an equivalence relation and the sets $B_{\left[ i_{1},i_{2}\right]
}\left( x\right) $ are the equivalence classes.

\begin{lemma}
If $x\in A^{\mathbb{Z}}$ and $n\in \mathbb{N}^{\ast }$ then :\newline
(i) $B_{\left[ -n,n\right] }\left( x\right) $ is closed.\newline
(ii) $F\left( B_{\left[ -n,n\right] }\left( x\right) \right) \sqsubseteq B_{%
\left[ -n,n\right] }\left( F\left( x\right) \right) .$
\end{lemma}

\begin{definition}
Let $\left( F,\mu \right) $ a cellular automaton equipped with a shift
ergodic measure $\mu ,$ a point $x$ is $\mu -$equicontinuous if for any $m>0$
we have :%
\begin{equation*}
\underset{n\rightarrow \infty }{\lim }\frac{\mu \left( \left[ x\left(
-n,n\right) \right] \cap B_{\left[ -m,m\right] }\left( x\right) \right) }{%
\mu \left( \left[ x\left( -n,n\right) \right] \right) }=1.
\end{equation*}

We say that $F$ is $\mu -$almost expansive if there exist $m>0$ such that
for all $x\in A^{\mathbb{Z}}:\mu \left( B_{\left[ -m,m\right] }\left(
x\right) \right) =0.$
\end{definition}

\begin{definition}
Let $\left( F,\mu \right) $ denote a cellular automaton equipped with a
shift ergodic measure $\mu .$ Define classes of cellular automata as follows
:\newline
1- $\left( F,\mu \right) \in \mathcal{A}$ if $F$ is equicontinuous at some $%
x\in A^{\mathbb{Z}}.$\newline
2- $\left( F,\mu \right) \in \mathcal{B}$ if $F$ is $\mu -$almost
equicontinuous at some $x\in A^{\mathbb{Z}}$ but $F\notin \mathcal{A}$.%
\newline
3- $\left( F,\mu \right) \in \mathcal{C}$ if $F$ is $\mu -$almost expansive.
\end{definition}

The following proposition and its proof is a result due to Gilman \cite%
{Gil88} the proof was detailed in \cite{Chem}. 

\begin{proposition}
\label{Gil} Let $F$ be a cellular automaton or radius $r$, If $F$ has an
equicontinuity point $x$ then the sequence of words $\left( F^{i}\left(
x\right) \left( -r,r\right) \right) _{i\geq 0}$ is eventually periodic.
\end{proposition}

\begin{proof}
Let $r$ be the radius of the cellular automaton and let $x$ be a $\nu $%
-equicontinuous point. Then we have $\nu \left( B_{\left[ -r,r\right]
}\left( x\right) \right) >0.$\newline

By the ergodicity of the shift, for a fixed value $p>0$, we have:%
\begin{equation*}
\nu \left( \sigma ^{-p}B_{\left[ -r,r\right] }\left( x\right) \cap \left( B_{%
\left[ -r,r\right] }\left( x\right) \right) \right) >0.
\end{equation*}%
For any $y\in \sigma ^{-p}B_{\left[ -r,r\right] }\left( x\right) \cap \left(
B_{\left[ -r,r\right] }\left( x\right) \right) $ we have:%
\begin{equation*}
\forall k\geq 0:F^{k}\left( y\right) \left( -r,r\right) =F^{k}\left(
y\right) \left( -r+p,r+p\right) .
\end{equation*}

Let us denote $y\left( -r,r\right) =w$ and by $u$ the word between two
occurrences of $w$. Let us also denote by $y^{-}$ the part at left of the
first word $w$ and $y^{+}$ that's at right of the second $w.$

By incorporating the word $uw$ in $y$ between the words $w$ and $u$ we
obtain a new element $y^{\left( 1\right) }$ containing two occurrences of
the word $uw.$

By repeating this process we obtain a sequence $y^{\left( i\right) }$
containing at each iteration one more occurrence of the word $uw.$

By recurrence under the $F-$ action we can show that that $y^{\left(
i\right) }$ still belongs to $B_{\left[ -ij_{1}-\left( i+1\right)
r,ij_{1}+\left( i+1\right) r\right] }\left( x\right) $ for any $i>0.$ 
\newline
For $i=1$ we show that $y^{\left( 1\right) }\in B_{\left[ -j_{1}+2r,j_{1}+2r%
\right] }\left( x\right) .$\newline
As images of the word $y^{\left( 1\right) }\left( j_{1},j_{1}+2r\right) =w$
are independent at right from the infinite column under the word $y^{+}$ and
at left from the infinite column under $wu$. From another side images of the
word $y^{\left( 1\right) }\left( -j_{1}-2r,-j_{1}\right) =w$ depend at left
from the infinite column under $y^{-}$ and at right from the column under $%
wu $.

From the definition of the set $B_{\left[ -r,r\right] }\left( x\right) ,$
images of $x\left( -r,r\right) =w$ are independent from the behavior at left
and right of $w.$So we have : 
\begin{equation*}
\forall i:F^{i}\left( y^{\left( 1\right) }\right) \left( -r,r\right)
=F^{i}\left( y^{\left( 1\right) }\right) \left( -r,r\right) .
\end{equation*}%
\newline
Then we have: $y^{\left( 1\right) }\in B_{\left[ -j_{1}+2r,j_{1}+2r\right]
}\left( x\right) .$

Suppose now that $y^{\left( i\right) }\in B_{\left[ -ij_{1}-\left(
i+1\right) r,ij_{1}+\left( i+1\right) r\right] }\left( x\right) $ we want to
show that $y^{\left( i+1\right) }\in B_{\left[ -\left( i+1\right)
j_{1}-\left( i+2\right) r,\left( i\right) i+1j_{1}+\left( i+2\right) r\right]
}\left( x\right) .$\newline
Same arguments as step $i=1$ lead to conclude that images of the word $%
y^{\left( i\right) }\left( ij_{1},ij_{1}+\left( i+1\right) r\right) =w$
depends at right from the infinite word $y^{+}$ and at left from the word $%
wu $ and from the infinite column below. On another hand images of the word $%
y^{\left( 1\right) }\left( -ij_{1},-ij_{1}-\left( i+1\right) r\right) =w$
depends at left from the infinite word $y^{-}$ and at right from $wu$ and
the infinite column below.\newline
As each element $y^{\left( i\right) }$ belongs to $B_{\left[ -ij_{1}-\left(
i+1\right) r,ij_{1}+\left( i+1\right) r\right] }\left( x\right) $ it belongs
also to $B_{\left[ -r,r\right] }\left( x\right) .$

The sequence of configurations $y^{\left( i\right) }=y^{-}w\underset{%
\leftarrow \text{ }\left( i-1\right) \text{ }\rightarrow }{uw...uw}y^{+}$
containing at each step a one more occurrence of the word $uw$ converge to
the shift periodic configuration $\left( uw\right) ^{\infty }$ which share
with $y$ the same coordinates over $\left( -r,r\right) $.\newline
As the set $B_{\left[ -r,r\right] }\left( x\right) $ is closed the periodic
point $\left( uw\right) ^{\infty }$ is in $B_{\left[ -r,r\right] }\left(
x\right) ,$ the sequence of words $\left( F^{i}\left( x\right) \left(
-r,r\right) \right) _{i\geq 0}$ is then eventually periodic.
\end{proof}

\subsection{Periodic points of cellular automata}

Let $(A^{\mathbb{Z}},F)$ be a cellular automaton. By commutation with the
shift, every shift-periodic point is $F-$eventually periodic hence the set
of eventually periodic points is dense.

Shift periodic points are called spatially periodic and periodic points of a
cellular automata that are not shift periodic are called \emph{strictly\
temporally\ periodic\ points.}

Boyle and Kitchens \cite{BK99} showed that closing cellular automata have a
dense set of periodic points. The same result was obtained by Blanchard and
Tisseur \cite{BT00} for surjective cellular automata with equicontinuity
points.

Acerbi, Dennunzio and Formenti \cite{ADF09} showed that if every mixing
cellular automaton has a dense set of periodic points then every cellular
automaton has a dense set of periodic points.

The question whatever a surjective cellular automaton has a dense set of
periodic points is still an open problem.

The two following results inspired our proposition on the density of \emph{%
strictly\ temporally\ periodic\ points }when the\emph{\ }cellular automaton
has almost equicontinuous points.

\begin{proposition}[\protect\cite{Tis08}]
Let $(A^{\mathbb{Z}},F,\mu )$ be a cellular automaton and $\mu $ an
invariant measure; if $F$ has $\mu -$equicontinuous points then the set of
periodic points is dense in the topological support of $\mu .$
\end{proposition}

\begin{proposition}[\protect\cite{Len12}]
Let $(A^{\mathbb{Z}};F)$ be an almost equicontinuous and surjective CA.
Then, the set of \emph{strictly temporally periodic points} of $F$ is dense.
\end{proposition}

\begin{proposition}[\protect\cite{Len12}]
Let $(A^{\mathbb{Z}};F)$ be an expansive CA. Then, the set of \emph{strictly
temporally periodic points} of $F$ is empty.
\end{proposition}

\section{Results}

\subsection{Strictly temporally periodic points}

In this section we extend a result of Lena, Margara and Dennunzio about the
density of the set of strictly temporally periodic to cellular automata with
almost equicontinuous points.

\begin{proposition}
Let $(A^{\mathbb{Z}},F)$ be a cellular automaton and $\mu $ an $F-$invariant
and shift ergodic measure.\newline
If $F$ has almost equicontinuous points then the set of strictly temporally
periodic points of $F$ is dense in the topological support of $\mu .$
\end{proposition}

\begin{proof}
Let $r$ be the radius of the cellular automaton and $z$ an almost
equicontinuous point, so $\mu (B_{[-r,r]}(z))>0$. Since $\mu $ is a shift
ergodic measure $\exists p>2r+1$ such that :%
\begin{equation*}
\mu (S=B_{[-r,r]}(z)\cap \sigma ^{-p}(B_{[-r,r]}(z)))>0.
\end{equation*}%
From the Poincar\'{e} recurrence theorem, there exists $m\in \mathbb{N^{\ast
}}$ and $x\in S$ such that\newline
$F^{m}(x)_{[-r,r+p]}=x_{[-r,r+p]}$. Set $%
w=z_{[-r,r]}=x_{[-r,r]}=x_{[-r+p,r+p]}$ and $u=x_{[r,-r+p]}$.\newline
Since $p-2r\geq 2$ $\exists v\neq u$ such that $|wv|=|wu|$.\newline
From the proof of proposition $\ref{Gil}$ the shift periodic points $%
a=(wv)^{\infty }$, $b=(wu)^{\infty }$ and the point $y$ such that\newline
$y_{]-\infty ,-r[}=(wv)^{\infty }$ $y_{[-r,+\infty \lbrack }=(wu)^{\infty }$
belongs to $S$.

Let us define the sequence $(y^{(i)})_{i\in \mathbb{N}}$ by: 
\begin{equation*}
\left\{ 
\begin{array}{l}
y^{(0)}=y_{[-r,r+p]}=wuw \\ 
y^{(i)}=y_{[-r-ip,r+(i+1)p]}=(wv)^{i}\;wuw\;(uw)^{i}%
\end{array}%
\right.
\end{equation*}%
\begin{equation*}
\FRAME{itbpF}{5.2849in}{2.5097in}{0in}{}{}{fig2.png}{\special{language
"Scientific Word";type "GRAPHIC";maintain-aspect-ratio TRUE;display
"USEDEF";valid_file "F";width 5.2849in;height 2.5097in;depth
0in;original-width 8.7605in;original-height 4.1355in;cropleft "0";croptop
"1";cropright "1";cropbottom "0";filename 'Fig2.PNG';file-properties
"XNPEU";}}
\end{equation*}%
We need to show by induction that $y^{(i)}\in B_{[-r-(i-1)p,r+ip]}(y)$ $%
\forall i\in \mathbb{N^{\ast }}$\newline
For $i=1$ , Since $y_{[-r-p,r+2p]}^{(1)}=wv\;wuw\;uw$ so $%
y_{[-r,r+p]}^{(1)}=y_{[-r,r+p]}=wuw$. Suppose that for $j>0$ one has $%
F^{t}(y^{(1)})_{[-r,r+p]}=F^{t}(y)_{[-r,r+p]}$ for $0\leq t\leq j$. In this
case 
\begin{align*}
F^{j+1}(y^{(1)})_{[-r,r+p]}& =f(F^{j}(y^{(1)})_{[-2r,2r+p]}) \\
& =f(F^{j}(y^{(1)})_{[-2r,-r]}F^{j}(y)_{[-r,r+p]}F^{j}(y^{(1)})_{[r+p,r+2p]})
\\
& =F^{j+1}(y)_{[-r,r+p]}
\end{align*}%
We can conclude that $(F^{j}(y^{(1)})_{[-r,r+p]})_{j\in \mathbb{N}%
}=(F^{j}(y)_{[-r,r+p]})_{j\in \mathbb{N}}$ which implies that $y^{(1)}\in
B_{[-r,r+p]}(y)$\newline
Next, assuming that $y^{(i-1)}\in B_{[-r-(i-2)p,r+(i-1)p]}(y)$ and using the
same arguments as in step $i=1$ we show that $%
(F^{j}(y^{(i)})_{[-r-(i-1)p,r+ip]})_{j\in \mathbb{N}%
}=(F^{j}(y)_{[-r-(i-1)p,r+ip]})_{j\in \mathbb{N}}$ which implies that $%
y^{(i)}\in B_{[-r-(i-1)p,r+ip]}(y)$.\newline
Since for all $i\in \mathbb{N}$ $y^{(i)}\in B_{[-r-(i-1)p,r+ip]}(y)$ it
follows that :%
\begin{equation*}
y^{(i)}\in B_{[-r,r+p]}(y)=B_{[-r,r]}(z)\cap \sigma ^{-p}(B_{[-r,r]}(z))=S
\end{equation*}%
The sequence $(y^{i})_{i\in \mathbb{N}}$ of points of $S$ converge to $%
y^{^{\prime }}=(wv)^{\infty }wuw(uw)^{\infty }$. Since $S$ is closed $%
y^{^{\prime }}\in S$ and since the $F$ orbit of each point in $S$ share the
same central coordinates, it follows that $F^{m}(y^{^{\prime
}})_{[-r,r+p]}=x_{[-r,r+p]}=y_{[-r,r+p]}^{^{\prime }}=wuw$ which implies
that $F^{m}(y^{^{\prime }})_{[-r,+\infty \lbrack }=y_{[-r,+\infty \lbrack
}^{^{\prime }}$.\newline
Using the same arguments we show that there exists a point $x^{^{\prime
}}\in B_{[-r-p,r-p]}\cap \sigma ^{-p}(B_{[-r-p,r-p]})$ such that 
\begin{equation*}
F^{m}(y^{^{\prime }})_{[-r-p,r]}=x_{[-r-p,r]}^{^{\prime
}}=y_{[-r-p,r]}^{^{\prime }}=wvw
\end{equation*}%
which implies that $F^{m}(y^{^{\prime }})_{]-\infty ,-r]}=y_{]-\infty
,-r]}^{^{\prime }}$. It follows that $F^{m}(y^{^{\prime }})=y^{^{\prime }}$
and permit to conclude.
\end{proof}

\begin{example}[\protect\cite{Gil87}]
Consider the following cellular automaton $\left( \mathbf{B}^{\mathbb{Z}%
},F\right) $ here $\mathbf{B=}\left\{ \square ,\swarrow ,\downarrow \right\} 
$\newline
Here the particle $\square $ is considered as background , the particle $%
\swarrow $ moves to the left and the particle $\downarrow $ moves downward.%
\newline
When a $\swarrow $ and $\downarrow $ meets the collision produce a $\square $
(mutual annihilation). The local rule is given in the following table 
\begin{equation*}
\begin{tabular}{|c|c|c|c|c|c|c|c|c|}
\hline
$\ast \square \square $ & $\ast \square \swarrow $ & $\ast \square
\downarrow $ & $\ast \swarrow \square $ & $\ast \swarrow \swarrow $ & $\ast
\swarrow \downarrow $ & $\ast \downarrow \square $ & $\ast \swarrow
\downarrow $ & $\ast \downarrow \downarrow $ \\ \hline
$\square $ & $\swarrow $ & $\square $ & $\square $ & $\swarrow $ & $\square $
& $\downarrow $ & $\square $ & $\downarrow $ \\ \hline
\end{tabular}%
\end{equation*}%
This cellular automaton is in category $\mathcal{B}$ hence sensitive.
\end{example}

\subsection{Periodic factors}

\begin{proposition}
Let $(A^{\mathbb{Z}},F)$ be a cellular automaton with equicontinuous points
then $F$ has as factor at least a periodic factor.
\end{proposition}

\begin{proof}
As $F$ has some equicontinuous point then there is a class $B_{\left[ -n,n%
\right] }\left( x\right) $ with $B_{\left[ -n,n\right] }\left( x\right)
^{\circ }\neq \varnothing .$\newline

By lemma \ref{Gil} the sequence of words $\left( F^{i}\left( x\right) \left(
-r,r\right) \right) _{i\geq 0}$ is eventually periodic, then there exist a
preperiod $m$ and a period $p$ such that for all $i\in \mathbb{N}:$ 
\begin{equation*}
F^{m}\left( x\right) \left( -r,r\right) ^{\infty }=F^{m+ip}\left( x\right)
\left( -r,r\right) ^{\infty }.
\end{equation*}%
\newline
Let us define the family of sets $W_{k}=\left[ F^{k}\left( x\right) \left(
-r,r\right) \right] $ with $0\leq k\leq m+p-1.$

Consider the set $W=\cup _{m}^{m+p-1}W_{k}$ we will show that the
restriction of $F$ to $W$ has a periodic factor.

Let the function $\pi $ defined from $W$ to $\mathbb{Z}/p\mathbb{Z}$ by :%
\begin{equation*}
\pi \left( x\right) =k-m:x\in W_{k}:m\leq k\leq m+p-1
\end{equation*}%
Define the periodic dynamical systems $\left( \mathbb{Z}/p\mathbb{Z},\left(
x+1\right) \func{mod}p\right) ,$ we have 
\begin{equation*}
\pi \left( F\left( x\right) \right) =\left\{ 
\begin{array}{l}
k-m+1\text{ if }x\in W_{k}\text{ and }m\leq k\leq m+p-2 \\ 
1\text{if }x\in W_{m+p-1}%
\end{array}%
\right. =P\left( \pi \left( x\right) \right)
\end{equation*}%
Thus $\left( \mathbb{Z}/p\mathbb{Z},\left( x+1\right) \func{mod}p\right) $
is a periodic factor of $\left( W,F\right) .$
\end{proof}

\begin{example}
Non surjective CA with equicontinuous points \newline
$\left( \mathbf{3}^{\mathbb{Z}},F\right) $ where $\mathbf{3}=\left\{
w,0,r\right\} $ the rules are defined by 
\begin{equation*}
\begin{tabular}{|l|l|l|l|l|l|l|l|l|l|}
\hline
$x_{i-1}x_{i}$ & $wr$ & $w0$ & $ww$ & $rr$ & $r0$ & $rw$ & $0r$ & $00$ & $0w$
\\ \hline
$F\left( x\right) _{i}$ & ~~$r$ & ~~$0$ & ~~$w$ & ~$0$ & ~$r$ & ~$w$ & ~$0$
& ~$0$ & ~$w$ \\ \hline
\end{tabular}%
\end{equation*}%
This cellular automaton is not surjective furthermore we have $F\left( \left[
w000w\right] \right) =F\left( \left[ w00rw\right] \right) =\left[ w000w%
\right] $ hence $\left( \mathbf{3}^{\mathbb{Z}},F\right) $ is not injective.%
\newline
Consider the cylinder $\left[ wr000w\right] $ we have 
\begin{equation*}
\begin{tabular}{l|l|llllll|l}
$x$ & $...$ & $w$ & $r$ & $0$ & $0$ & $0$ & $w$ & $...$ \\ 
$F\left( x\right) $ & $...$ & $w$ & $r$ & $r$ & $0$ & $0$ & $w$ & $...$ \\ 
$F^{2}\left( x\right) $ & $...$ & $w$ & $r$ & $0$ & $r$ & $0$ & $w$ & $...$
\\ 
$F^{3}\left( x\right) $ & $...$ & $w$ & $r$ & $r$ & $0$ & $r$ & $w$ & $...$
\\ 
$F^{4}\left( x\right) $ & $...$ & $w$ & $r$ & $0$ & $r$ & $0$ & $w$ & $...$%
\end{tabular}%
\end{equation*}%
\newline
Denote by $W=\left[ wr0r0w\right] \cup \left[ wrr0rw\right] $ then the
restriction of $F$ to the set $W$ has a periodic factor.
\end{example}

\begin{proposition}
Let $(A^{\mathbb{Z}},F)$ be a surjective cellular automaton with
equicontinuous points but without being equicontinuous; then the set of
equicontinuous factors contain an infinite union of equivalence classes $%
\widetilde{\mathit{p}}$.
\end{proposition}

\begin{proof}
Let us suppose that the set of equicontinuous factors contain a finite
number of equivalence classes $\widetilde{\mathit{p}}$.\newline
Denote by $\widetilde{\mathit{p}}_{1},\widetilde{\mathit{p}}_{2},...%
\widetilde{\mathit{p}}_{s}$ the existing $s$ equivalence classes.

Let $r$ be the radius of the cellular automaton and let $x$ be an
equicontinuous point.

Let $y\in A^{\mathbb{Z}}$ from the proof of Lemma \ref{Gil} we know that the
sequence 
\begin{equation*}
F^{i}\left( x\left( -r,r\right) y\left( -n,n\right) x\left( -r,r\right)
\right)
\end{equation*}
is eventually periodic hence associated to a periodic factor.

Let $P=\underset{1\leq i\leq s}{\func{lcm}}\left( \widetilde{\mathit{p}}%
_{i}\right) $ and the sequence of the $F-$periodic points \newline
$\left( x\left( -r,r\right) y\left( -n,n\right) x\left( -r,r\right) \right)
^{\infty }$ that converges to $y.$

We have then for any $y$ in $A^{\mathbb{Z}}$ : 
\begin{eqnarray*}
F^{P}\left( y\right)  &=&F^{P}\left( \underset{k\rightarrow \infty }{\lim }%
x\left( -r,r\right) y\left( -n,n\right) x\left( -r,r\right) ^{\infty
}\right) =\underset{k\rightarrow \infty }{\lim }F^{P}\left( x\left(
-r,r\right) y\left( -n,n\right) x\left( -r,r\right) ^{\infty }\right)  \\
&=&\underset{n\rightarrow \infty }{\lim }\left( x\left( -r,r\right) y\left(
-n,n\right) x\left( -r,r\right) ^{\infty }\right) =y.
\end{eqnarray*}%
\newline
We obtain then $F^{P}=Id$ and hence $F$ is equicontinuous which is a
contradiction.
\end{proof}

\begin{corollary}
The maximal equicontinuous factor of a surjective cellular automaton with
equicontinuity points but without being equicontinuous is not a cellular
automaton.
\end{corollary}

\begin{proof}
Suppose that the maximal equicontinuous factor is a cellular automaton and
denote it by $M$. It is then surjective and there exist a period $p$ such
that $M^{p}=M.$

As there is an infinity of surjective periodic factors classes. Choose a
periodic factor $E$ such that $q$ the period of $E$ satisfy $q>p$ as $E$ is
a factor of $M$ then $q$ is a divisor of $p$ which is a contradiction.
\end{proof}

\section{Conclusion}

We have shown that any cellular automaton with almost equicontinuous points
has a dense set of strictly temporally periodic points in the topological
support of the measure.

Notice that this category contains only cellular automata without
topological equicontinuous points hence sensitive cellular automata.

It was shown in \cite{KZ13} that there exist mixing cellular automata with a
dense set of strictly temporally periodic points. There exist also mixing
cellular automata whose set of strictly temporally periodic points is
neither empty nor dense.

The dual question for Gilman's classification is interesting . Is there an
example of almost expansive cellular automaton such that the set of strictly
temporally periodic points is neither empty nor dense ?

\end{document}